 \documentclass{amsart}
\usepackage{amssymb,amsmath}
\usepackage{amsfonts}

\makeindex

\usepackage[italic]{sublabel}
\usepackage{color}
\definecolor{medgray}{gray}{0.6}

\usepackage[ruled]{algorithm}
\usepackage{algpseudocode}

\newcommand{\ignore}[1]{}
\newcommand{\xx}{\mathbf x}
\newcommand{\dd}{\mathbf d}
\newcommand{\ttt}{{\lfloor t/2\rfloor}}

\newcommand{\PR}{\oR[\xx]}
\newcommand{\QPR}{\oR[\xx]/I}
\newcommand{\DQPR}{(\oR[\xx]/I)^*}

\newcommand{\oT}{\mathbb T}
\newcommand{\Ker}{{\rm Ker\,}}
\newcommand{\Span}{{\rm Span}}

\newcommand{\rank}{{\rm rank}}
\newcommand{\vect}{{\rm vec}}
\newcommand{\oR}{\mathbb R}
\newcommand{\oK}{\mathbb K}

\newcommand{\oN}{\mathbb N}
\newcommand{\oC}{\mathbb C}

\newcommand{\BB}{{\mathcal B}}
\newcommand{\CC}{{\mathcal C}}
\newcommand{\DD}{{\mathcal D}}

\newcommand{\GG}{{\mathcal G}}
\newcommand{\HH}{{\mathcal H}}
\newcommand{\KK}{{\mathcal K}}

\newcommand{\NN}{{\mathcal N}}
\newcommand{\NNN}{{\mathcal N}}
\newcommand{\PP}{{\mathcal P}}

\newcommand{\SSS}{{\mathcal S}}
\newcommand{\WW}{{\mathcal W}}
\newcommand{\XX}{{\mathcal X}}

 \newtheorem{theorem}{\sc Theorem}
 \newtheorem{definition}[theorem]{\sc Definition}
 \newtheorem{lemma}[theorem]{\sc Lemma}
 \newtheorem{example}[theorem]{\sc Example}

 \newtheorem{proposition}[theorem]{\sc Proposition}

 \newtheorem{corollary}[theorem]{\sc Corollary}
  \newtheorem{remark}[theorem]{\sc Remark}

\begin{document}



\title[Real and complex prolongation-projection algorithms]{A prolongation-projection algorithm for computing the finite real variety of an ideal}

\author{Jean Lasserre}
\address{LAAS-CNRS\\7 Avenue du Colonel Roche\\31 077 Toulouse, France}
\curraddr{}
\email{lasserre@laas.fr}

\author{Monique Laurent}
\address{CWI\\P.O. Box 94079\\1090 GB Amsterdam, The Netherlands}
\email{M.Laurent@cwi.nl}

\author{Philipp Rostalski}
\address{Automatic Control Lab.\\ETH Zurich\\Physikstrasse 3\\8092 Zurich, Switzerland}
\email{rostalski@control.ee.ethz.ch}



\begin{abstract}
We provide a real algebraic symbolic-numeric algorithm for computing the real variety
$V_\oR(I)$ of an ideal $I\subseteq \PR$, assuming $V_\oR(I)$ is finite (while $V_\oC(I)$ could be infinite). 
Our approach uses sets of linear functionals on $\PR$, vanishing on a given set of polynomials generating $I$ and their prolongations up to a given degree, as well as on polynomials of the real radical ideal  $\sqrt[\oR] I$ obtained 
from the kernel of a suitably defined moment matrix assumed to be positive semidefinite and of maximum rank.
 We formulate a condition on the dimensions of projections of these sets of linear functionals, 
 which serves as stopping criterion for our algorithm; this new criterion is satisfied earlier than the previously used stopping criterion based on a rank condition for moment matrices.
This algorithm is based on standard numerical linear algebra routines and semidefinite optimization and
combines techniques from previous work of the authors together with an existing algorithm for the complex variety. 


\end{abstract}

\maketitle



\section{Introduction}
\label{sec::intro}
Polynomial equations play a crucial role in mathematics and are widely used in an emerging number of modern applications. Recent years have witnessed a new trend in algebraic geometry and polynomial system,
namely numerical polynomial algebra~\cite{St04} or numerical algebraic geometry~\cite{SW05}. Algorithms in this field deal with the problem of (approximately) computing objects of interest in the classical area of algebraic geometry with a focus on polynomial root finding. 

There is a broad literature for the problem of computing complex roots, 
that deals with numerical and symbolic algorithms, ranging from numerical continuation methods as in e.g. Verschelde ~\cite{Ve99} to exact methods as in e.g. Rouillier \cite{Rouillier1999a}, or more general Gr\"obner or border bases methods; see e.g. the monograph \cite{DiEm05} and the references therein.

In many practical applications, one is only interested in the {\it real}
 solutions of a system of polynomial equations, possibly satisfying additional polynomial inequality constraints. 
An obvious approach for finding all {real} roots of a system of polynomial equations is to first compute all {complex} solutions, i.e., the algebraic variety $V_\oC(I)$ of the associated ideal $I\subseteq\oR[\xx]$, and then to sort the real variety $V_\oR(I)=\oR^n \cap V_\oC(I)$ from $V_\oC(I)$ afterwards. 
However, in many practical instances, the number of real roots is considerably smaller than the total number of roots and, in some cases, it is finite 
while $\vert V_\oC(I)\vert = \infty$.

The literature about algorithms tailored to the problem of {real solving} systems of 
polynomial equations is by far not as broad as for the problem of computing complex roots.
Often local Newton type methods or subdivision methods based on Descartes rule of sign, on Sturm-Habicht sequences or on Hermite quadratic forms are used; see e.g.~\cite{Basu2003,Mourrain2005a,PMRQT07} for a discussion. 
In \cite{LLR07} we gave  an algorithm for finding $V_\oR(I)     $ (assumed to be finite),
and  a semidefinite characterization as well as a border (or Gr\"obner) basis of the real radical
 ideal $\sqrt[\oR]I$, by using linear algebra combined with semidefinite programming (SDP) techniques. 
 We exploited the fact that {all} information needed to compute the above objects is contained in the so-called {moment matrix} (whose entries depend on the polynomials generating the ideal $I$) and the geometry behind when this matrix is required to be positive semidefinite with maximum rank.
 We use the name {\em (real-root) moment-matrix algorithm}
 for the algorithm proposed in \cite{LLR07}. This algorithm was later extended to the computation of all complex roots in 
\cite{LLR08}. 
A feature of the real-root moment-matrix algorithm is that it requires to solve a sequence of SDP problems involving matrices of increasing size until a certain rank condition is satisfied. 
Solving the SDP problem is the computationally most demanding task in the algorithm. It is thus important to be able
to terminate the algorithm at an as early as possible stage 
so that the size of the matrices does not grow too much.
This is the motivation for the present paper where we present a new stopping condition, which is satisfied at least as early as the rank condition of \cite{LLR07} (and often earlier on examples).
This leads to a new algorithm which we name  {\em (real-root)
 prolongation-projection algorithm} since
its stopping condition involves computing the dimensions      of projections of certain sets of linear functionals on spaces of polynomials. 
This new algorithm arises by incorporating several ideas of \cite{LLR07,LLR08}  into an existing symbolic-numeric solver dedicated to compute $V_\oC(I)$ (as described  e.g. in~\cite{ReZh04}). A detailed description will be given in Section \ref{sec::algorithm} but, in order
to ease comparison with the moment-matrix method of \cite{LLR07},
we now give  
 a brief sketch of both methods.

\subsection*{Sketch of the real-root moment-matrix and prolongation-projection algorithms}
While methods based on Gr\"obner bases work with the (primal) ring of polynomials $\PR$, its ideals and their associated quotient spaces, we follow a dual approach here. The algorithms proposed in \cite{LLR07} and in this work manipulate specific subspaces of $(\PR)^*$, the space of linear forms dual to the ring of multivariate polynomials.

We denote by $(\PR_t)^*$ the space of linear functionals on the set $\PR_t$ of polynomials with degree at most $t$ and use the notion of {\em moment matrix} 
$M_s(L):=(L(\xx^\alpha\xx^\beta))$ (indexed by monomials of degree at most $s$) for 
$L\in (\PR_{2s})^*$. (See Section \ref{sec::preliminaries} for more definitions.)
Say we want to compute the (finite) real variety $V_\oR(I)$ of an ideal $I$ given by a set of generators 
$h_1,\ldots,h_m\in\PR$ with maximum degree $D$.
A common step in both methods is to compute a maximum rank moment matrix $M_\ttt(L)$, where $L\in (\PR_t)^*$ 
vanishes on the set $\HH_t$ of all prolongations up to degree $t$ of the polynomials $h_j$; this step
is carried out with a numerical algorithm for semidefinite 
 optimization. From that point on both methods use distinct strategies. 
In the moment-matrix method one checks 
whether the  rank condition: $\rank M_s(L)=\rank M_{s-1}(L)$ 
holds for some 
 $D\le s\le \ttt$; if so, then one can conclude that $\sqrt[\oR]I$ is generated by the polynomials in the kernel
of $M_s(L)$ and extract $V_\oR(I)$; if not, iterate with $t+1$.
In the prolongation-projection algorithm,
one considers $\GG_t$, the set  obtained by adding to $\HH_t$ prolongations of the polynomials in the kernel of $M_\ttt(L)$, 
its border $\GG_t^+:=\GG_t \cup_i \xx_i\GG_t$, as well as 
the set $\GG_t^\perp$ of linear functionals on $\PR_t$ vanishing on $\GG_t$,
 and its projections $\pi_s(\GG_t^\perp)$ on various degrees $s\le t$.
We give conditions on the dimension of these linear subspaces ensuring the computation of the real variety $V_\oR(I)$ and generators for the real radical ideal $\sqrt[\oR]I$.
Namely, if $\dim \pi_s(\GG_t^\perp)=\dim\pi_{s-1}(\GG_t^\perp)=\dim\pi_s((\GG_t^+)^\perp)$ holds for some $D\le s\le t$, then one can compute an ideal $J$ nested between $I$ and $\sqrt [\oR]I$ so that $V_\oR(I)=V_\oR(J)$, with equality $J=\sqrt[\oR]I$ if $\dim \pi_s(\GG_t^\perp)=|V_\oR(I)|$; if not, iterate with $t+1$.

Both algorithms are tailored to finding real roots and terminate assuming $V_\oR(I)$ finite (while $V_\oC(I)$ could be infinite).
However, the order $t$ at which the dimension condition holds is at most the order at which the rank condition holds. Hence
 the prolongation-projection algorithm terminates earlier than the moment-matrix method, which often permits to save some semidefinite optimization step with a larger moment matrix (as shown on a few examples in Section \ref{sec::numexamples}).

\subsection*{Contents of the paper}
Section \ref{sec::preliminaries} provides some basic background on polynomial ideals and moment matrices whereas Section \ref{sec::principles} presents the basic principles behind the prolongation-projection method and
Theorem~\ref{theoZR}, our main result, provides a new stopping criterion for the computation of $V_\oR(I)$. Section \ref{sec::links} relates the prolongation-projection algorithm 
to the moment-matrix method of~\cite{LLR07}. In particular, Proposition \ref{prop2} shows that the rank condition used as stopping criterion in the moment-matrix method is equivalent to a strong version of the new stopping criterion;
as a consequence the new criterion is satisfied at least as early as the rank condition (Corollary \ref{corlink}).
Section \ref{sec::algorithm} contains  a detailed description of the algorithm
whose behavior is illustrated on a few examples in Section \ref{sec::numexamples}.

\section{Preliminaries}
\label{sec::preliminaries}
\subsection{Polynomial ideals and varieties}
We briefly introduce some notation and preliminaries for polynomials 
used throughout the paper and refer e.g. to \cite{CLO97} and \cite{CLO98} for more details.


Throughout $\PR:=\oR[x_1,\ldots,x_n]$ is the
ring of real polynomials in the $n$ variables $\xx=(x_1,\ldots,x_n)$ and 
 $\PR_t$ is the subspace of polynomials of degree at most $t\in \oN$.
For $\alpha\in \oN^n$, $\xx^\alpha=x_1^{\alpha_1}\cdots x_n^{\alpha_n}$ is the monomial with exponent $\alpha$ and degree
$|\alpha|=\sum_i\alpha_i$. For an integer $t\ge 0$,
 the set $\oN^n_t=\{\alpha\in\oN^n\mid |\alpha|\le t\}$ corresponds to the set of monomials of degree at most $t$, and
$\oT^n=\{\xx^\alpha\mid \alpha\in\oN^n\}, \ \ \oT^n_t=\{\xx^\alpha\mid \alpha\in \oN^n_t\}$
 denote the set of all monomials and of all monomials of degree at most $t$, respectively.
Given $S\subseteq \PR$, set $x_i S:=\{x_ip\mid p\in S\}$. The set
$$S^+:=S\cup x_1S\cup \ldots \cup x_nS$$
denotes the one degree \emph{prolongation} of $S$ and, for $\BB\subseteq \oT^n$,
$\partial \BB := \BB^+\setminus\BB$ is called the set of \emph{border} monomials of $\BB$.
The set $\BB$ is said to be {\em connected to 1} if any $m\in \BB$ can be written as 
$m=m_1\cdots m_k$ with $m_1=1$ and $m_1\cdots m_h\in \BB$ for all $h=1,\ldots,k$. For instance, $\BB$ is connected to 1 if it is closed under taking divisions, i.e. $m\in\BB$ and $m'$ divides $m$ implies $m'\in\BB$.

\medskip
Given $h_1, \ldots, h_m \in\PR$, $I=(h_1,\ldots,h_m)$ 
%
%
is the ideal generated by $h_1,\ldots,h_m$,
its algebraic variety is 
%
\begin{align*}
 V_{\oC}(I) := \left\{ v \in \oC^n \mid h_j(v) = 0\; \forall j=1,\ldots,m) \right\}
\end{align*}
and its real variety is $V_\oR(I):=\oR^n \cap V_{\oC}(I)$. The ideal $I$ is zero-dimensional when $V_\oC(I)$ is finite. The vanishing ideal of a set $V\subseteq \oC^n$ is the ideal 
$$I(V):=\{f\in \PR\mid f(v)=0\ \forall v\in V\}.$$
The Real Nullstellensatz (see e.g. \cite[\S 4.1]{CLO97}) asserts that $I(V_\oR(I))$ coincides with $\sqrt[\oR]I$, the real radical of $I$ defined as
$$\sqrt[\oR]I:=\Big\{p\in \PR\:\big|\: p^{2m} +\sum_j q_j^2\in I \ \text{ for some }q_j\in \PR, m\in \oN\setminus \{0\}\Big\}.$$
%
%
%
%


%
Given a vector space $A$ on $\oR$, its dual vector space is the space $A^*={\rm Hom}(A,\oR)$ 
consisting of all linear functionals from $A$ to $\oR$.
Given $B\subseteq A$, set $B^\perp:=\{L\in A^*\mid L(b)=0 \ \forall b\in B\}$, and $\Span_\oR(B):=\{\sum_{i=1}^m\lambda_i b_i\mid \lambda_i\in\oR, b_i\in B\}$.
Then $\Span_\oR(B)\subseteq (B^\perp)^\perp$, with equality when $A$ is finite dimensional.

For an ideal $I\subseteq \PR$, the space 
$\DD[I]:=I^\perp= \{L\in (\PR)^*\mid L(p)=0 \ \forall p\in I\},$
considered e.g. by Stetter \cite{St04}, is isomorphic to 
$(\PR/I)^*$ and $\DD[I]^\perp=I$ when $I$ is zero-dimensional.
Recall that $I$ is zero-dimensional precisely when $\dim \PR/I<\infty$,
 and
$|V_\oC(I)|\le \dim \PR/I$ with equality precisely when $I=I(V_\oC(I))$.
%
%
%

The canonical basis of $\PR$ is the monomial set $\oT^n$, with
$\DD_n:=\{\dd_\alpha\mid\in\oN^n\}$ 
as  corresponding dual basis for $(\PR)^*$, where
$$\dd_\alpha(p) = \frac{1}{\prod_{i=1}^n \alpha_i!} 
\left(\frac{\partial^{|\alpha|}}{\partial x_1^{\alpha_1} \ldots \partial x_n^{\alpha_n}}p\right)(0)\ \ \text{ for } p\in\PR.$$
Thus any $L\in (\PR)^*$ can be written in the form
$L=\sum_\alpha y_\alpha \dd_\alpha$ (for some $y\in\oR^{\oN^n}$).

By restricting its domain to $\PR_s$, any linear form $L\in (\PR)^*$ gives a linear form $\pi_s(L)$ in $(\PR_s)^*$. Throughout we let $\pi_s$ denote this projection from $(\PR)^*$ (or from $(\PR_t)^*$ for any $t\ge s$) onto $(\PR_s)^*$.


\medskip
Given a zero-dimensional  ideal $I\subseteq \PR$,
a well known method for computing $V_\oC(I)$ is the so-called eigenvalue method which relies on the
following theorem relating the eigenvalues of the multiplication operators in $\PR/I$ to the points in $V_\oC(I)$. See 
e.g. \cite[Chapter 2\S4]{CLO98}.

\begin{theorem} \label{thm::muloperator}
 Let $I$ be a zero-dimensional ideal in $\PR$ and $h\in \PR$.
The eigenvalues of the multiplication operator 
$$\begin{array}{llll}
m_h: & \PR/I & \longrightarrow & \PR/I \cr
     &  p\mod I    & \mapsto & ph \mod I
\end{array}$$
are the evaluations $h(v)$ of the polynomial $h$ at the points $v \in V_\oC(I)$.
Moreover, given a basis $\BB$ of $\PR/I$, the eigenvectors of the matrix
of the adjoint operator of $m_h$ with respect to $\BB$ are 
the vectors 
$(b(v))_{b\in\BB}\in\oR^{|\BB|}$ (for all $v\in V_\oC(I)$).
\end{theorem}

The extraction of the roots via the eigenvalues of the multiplication operators requires the knowledge of a basis of $\PR/I$
and an algorithm for reducing a polynomial $p \in \PR$ modulo the ideal $I$ in order to construct the multiplication matrices. 
Algorithms using Gr\"obner bases
can be used to perform this reduction by implementing a polynomial division algorithm (see \cite[Chapter 1]{CLO97}) or, as we will do in this paper, 
generalized normal form algorithms using border bases (see  \cite{LLR08}, \cite{Mourrain2005}, \cite{St04} for details).

\subsection{Moment matrices}

Given $L\in (\PR)^*$, let $Q_L$ denote the quadratic form on $\PR$ defined by
$Q_L(p):=L(p^2)$ for $p\in \PR$. $Q_L$ is said to be positive semidefinite, written as $Q_L\succeq 0$, if $Q_L(p)\ge 0$ for all $p\in \PR$.
Let $M(L)$ denote the matrix associated with $Q_L$ in the canonical monomial basis of $\PR$, with $(\alpha,\beta)$-entry 
$L(\xx^\alpha \xx^\beta)$ for $\alpha,\beta\in \oN^n$, so that
$$Q_L(p)=\sum_{\alpha,\beta\in \oN^n}p_\alpha p_\beta L(\xx^\alpha \xx^\beta)= \vect(p)^T M(L)\vect(p),$$
where $\vect(p)$ is the vector of coefficients of $p$ in the monomial basis $\oT^n$.
Then $Q_L\succeq 0$ if and only if the matrix $M(L)$ is positive semidefinite.
For a polynomial $p\in\PR$, $p\in \Ker Q_L$ (i.e. $Q_l(p) = 0$ and so $L(pq)=0$ for all
$q\in\PR$) if and only if $M(L)\vect(p)=0$. Thus we may identify
$\Ker M(L)$ with a subset of $\PR$, namely we say that a polynomial
 $p\in \PR$ lies in $\Ker M(L)$ if $M(L)\vect(p)=0$. Then $\Ker M(L)$ is an ideal in $\PR$, which is real radical when 
$M(L)\succeq 0$ (cf. \cite{Lau05}, \cite{Moe04}).
For an integer $s\ge 0$, 
$M_s(L)$ denotes the principal submatrix of $M(L)$ indexed by $\oN^n_s$. 
Then, in the canonical basis of $\PR_s$, 
 $M_s(L)$ is the matrix of the restriction of $Q_L$ to $\PR_s$, and $\Ker M_s(L)$ can be viewed as a subset of $\PR_s$.
It follows from an elementary property of positive semi\-def\-inite matrices that 
\begin{align}\label{psdker}
 M_t(L)\succeq 0 \Longrightarrow \Ker M_t(L)\cap\PR_s =\Ker M_s(L) \ \ \text{ for }
1\le s\le t,
\end{align}
\begin{align}\label{psdkerb}
M_t(L),\ M_t(L')\succeq 0\Longrightarrow 
\Ker M_t(L+L')=\Ker M_t(L)\cap \Ker M_t(L').
 \end{align}
We now recall some results about moment matrices which played a central role in our previous work \cite{LLR07} and 
are used here again.

\begin{theorem}\label{theoCF}\cite{CF96}
Let $L\in (\PR_{2s})^*$. If $\rank M_s(L)=\rank M_{s-1}(L),$ then
there exists (a unique) $\tilde L\in (\PR)^*$ such that
$\pi_{2s}(\tilde L)=L$, $\rank M(\tilde L)=\rank M_s(L)$, and
$\Ker M(\tilde L)=( \Ker M_s(L))$.
\end{theorem}

\begin{theorem}\label{theoA}(cf. \cite{LLR07,Lau05})
Let $L\in (\PR)^*$. If $M(L)\succeq 0$ and $\rank M(L)=\rank M_{s-1}(L)$, then 
$\Ker M(L)=( \Ker M_s(L)) $ is a zero-dimensional real radical ideal 
and $|V_\oC(\Ker M(L))|=\rank M(L)$.
\end{theorem}

\section{Basic principles for the prolongation-projection algorithm}
\label{sec::principles}

We present here  the results  underlying 
the prolongation-projection algorithm for computing
$V_\oK(I)$, $\oK=\oR,\oC$.
The basic techniques behind this section originally stem from the treatment of partial differential equations.
Zharkov et al.~\cite{ZYB94,ZYB93} were the first to apply these techniques to polynomial ideals. 
Section \ref{secZR1} contains the main result (Theorem \ref{theoZR}).
The complex case is inspired from \cite{ReZh04} and was treated in \cite{LLR08}. 
The real case goes along the same lines, so we only give a brief  sketch of proof in 
Section \ref{secsketch}.
In Section \ref{secconcrete} 
we indicate a natural choice for the polynomial system $\GG$ involved 
in Theorem \ref{theoZR}, which is based on ideas of \cite{LLR07} and 
 will be used in the  prolongation-projection algorithm. 


\subsection{New stopping criterion based on prolongation/projection dimension conditions}\label{secZR1}

We state the main result on which the prolongation-projection algorithm is based. We give a unified 
formulation for both complex/real cases.

\begin{theorem} \label{theoZR} 
Let $I=(h_1,\ldots,h_m)$ be an ideal in $\PR$, $D=\max_j\deg(h_j)$ and $s,t$ be integers with
$1\le s\le t$. 
Let $\GG\subseteq \PR_t$ satisfying $h_1,\ldots,h_m\in \GG$ and 
$\GG\subseteq I$ (resp., $\GG\subseteq \sqrt[\oR] I$).
If $\dim \pi_s(\GG^\perp)=0$
 then $V_\oC(I)=\emptyset$ (resp., $V_\oR(I)=\emptyset$).
Assume now that $s\ge D$ and 
\begin{align}
\sublabon{equation}
\dim \pi_s(\GG^\perp) = \dim \pi_{s-1} (\GG^\perp), \label{ZRa}\\
\dim \pi_s(\GG^\perp) = \dim \pi_s((\GG^+)^\perp). \label{ZRb}
\end{align}
\sublaboff{equation}
Then there exists a set $\BB\subseteq \oT^n_{s-1}$ closed under taking 
 divisions
(and thus connected to 1) for which the following direct sum decomposition holds:
\begin{equation}\label{dims}
 \PR_{s}= \Span_\oR(\BB) \ \oplus \  (\PR_{s}\cap \Span_\oR(\GG)).
\end{equation}
Let $\BB\subseteq \oT^n_{s-1}$ be any set connected to 1 for which (\ref{dims}) holds, let $\varphi$ be the projection
from $\PR_s$ onto $\Span_\oR(\BB)$ along $\PR_{s}\cap \Span_\oR(\GG)$,
and let
$F_0:=\{\varphi(m)\mid m\in\partial \BB\}$, $J:=(F_0)$.
Then $\BB$ is a basis of $\PR/J$ and $F_0$ is a border basis of $J$.
Moreover:
\begin{itemize}
\item If  $\GG\subseteq I$ then $J=I$.

\item If $\GG\subseteq \sqrt[\oR] I$ then
\[V_\oR(I)\,=\,V_\oC(J)\cap\oR^n\,;\quad
J\cap\PR_s\,=\,\Span_\oR(\GG)\cap\PR_s\,;\quad
\pi_s(\DD[J])=\pi_s(\GG^\perp),\]
\noindent
and in addition, $J=\sqrt[\oR]I$ if $\dim\pi_s(\GG^\perp)=|V_\oR(I)|$.
\end{itemize}
\end{theorem}

This result is proved in \cite{LLR08} in the case when $\GG=\HH_t\subseteq I$, where 
\begin{equation}\label{setHt}
\HH_t:=\{\xx^\alpha h_j\mid |\alpha|+\deg(h_j)\le t,\ j=1,\ldots,m\}
\end{equation}
consists of all prolongations to degree $t$ of the generators $h_j$ of $I$. Note however that in \cite{LLR08} we did not prove the existence of $\BB$ closed under taking divisions; we include a proof in Section \ref{secsketch} below.

The proof for arbitrary $\GG\subseteq I$ is identical to the case $\GG=\HH_t$.
In the case $\GG\subseteq \sqrt [\oR] I$,  the  proof\footnote{Note that if we would apply the previous result to the ideal $J:=(I\cup \GG)$ and the set $\GG$, then we would reach the desired conclusion, but under the stronger assumption
$s\ge \max(D,D')$, where $D'$ is the maximum degree of a generating set for $\GG$.}
is essentially analogous (except for the last claim $J=\sqrt [\oR]I$ which is specific to the real case).
 We give a brief sketch of proof in the next section, since this enables us to point out the 
impact of the various assumptions and, moreover, some technical details are needed later in the presentation.

\subsection{Sketch of proof for Theorem \ref{theoZR}}\label{secsketch}

We begin with a lemma used to show the existence of $\BB$ closed by division in Theorem \ref{theoZR}.

\begin{lemma}\label{lemstable}
Let $Y$ be a matrix whose columns are indexed by $\oT^n_s$.
Assume
\begin{equation}\label{assY}
\forall \lambda\in \oR^{|\oT^n_{s-1}|}\ \ 
\sum_{a\in\oT^n_{s-1}}\lambda_a Y_a=0\Longrightarrow 
\sum_{a\in\oT^n_{s-1}}\lambda_a Y_{x_ia}=0,
\end{equation}
where $Y_a$ denotes the $a$-th column of $Y$.
Then there exists $\BB\subseteq \oT^n_s$ which is closed under taking
divisions and indexes a maximum linearly independent set
of columns of $Y$.
\end{lemma}

\begin{proof}
Order the monomials in $\oT^n_s$ according to a total degree monomial ordering $\preceq$. Let $\BB\subseteq \oT^n_s$ index a maximum linearly independent set of columns of $Y$, which is constructed using the greedy algorithm applied to the ordering $\preceq$ of the columns. Then, setting
$\BB_m:=\{m'\in\BB\mid m'\prec m\}$, $m\in\BB$ 
 precisely when  $\BB_m\cup\{m\}$ indexes a linearly independent set of columns of $Y$.
We claim that $\BB$ is closed under taking divisions. For this assume $m\in\BB$ and
$m=x_im_1$ with $m_1\not\in\BB$.
As $m_1\not\in \BB$, we deduce that 
$$Y_{m_1}=\sum_{a\in \BB_{m_1}}\lambda_a Y_a \ \ \text{ for some scalars } 
\lambda_a.$$
For $a\in \BB_{m_1}$, $a\prec m_1$ implies $x_ia\prec x_im_1=m$, i.e.,
$x_ia\in\BB_m$. Applying (\ref{assY}) we deduce that
$$Y_m=\sum_{a\in \BB_{m_1}}\lambda_a Y_{x_ia},$$
which gives a linear dependency of $Y_m$ with the columns indexed by $\BB_m$, contradicting 
$m\in\BB$.
\end{proof}

We now sketch the proof of Theorem \ref{theoZR}.
Set $N:=\dim \pi_{s-1}(\GG^\perp)$. If $N=0$ then $V_\oK(I)=\emptyset$ (for otherwise the evaluation at $v\in V_\oK(I)$ would give a nonzero element of
$\pi_{s-1}(\GG^\perp)$).
Let $\{L_1,\ldots,L_N\}\subseteq \GG^\perp$ for which 
$\{\pi_{s-1}(L_1),\ldots,\pi_{s-1}(L_N)\}$ is a basis of $\pi_{s-1}(\GG^\perp)$.
Let $Y$ be the $N\times |\oT^n_{s-1}|$ matrix with $(j,m)$-th entry $L_j(m)$
for $j\le N$ and $m\in\oT^n_{s-1}$.
We verify that $Y$ satisfies the condition (\ref{assY}) of Lemma \ref{lemstable} (replacing $s$ by $s-1$). For this note that
%
$\sum_{a\in\oT^n_{s-2}}\lambda_aY_a=0$ if and only if 
$p :=\sum_{a\in\oT^n_{s-2}}\lambda_a a \in (\pi_{s-2}(\GG^\perp))^\perp=
\Span_\oR(\GG)\cap\PR_{s-2}$ and thus
$x_ip\in  \Span_\oR(\GG^+)\cap\PR_{s-1}$; in view of (\ref{ZRb}), this implies 
$x_ip\in \Span_\oR(\GG)\cap \PR_{s-1}$ and thus  $\sum_{a\in\oT^n_{s-2}}\lambda_aY_{x_ia}=0$. 
Thus we can apply Lemma \ref{lemstable}: There exists a set $\BB$ indexing a maximum linearly independent set of columns of $Y$ which is closed by division. This amounts to having the direct sum decomposition:
\begin{equation}\label{dims-1}
\PR_{s-1}= \Span_\oR(\BB)\oplus (\Span_\oR(\GG)\cap\PR_{s-1})).
\end{equation}
As $N=\dim\pi_s(\GG^\perp)$, the set $\{\pi_s(L_1),\ldots,\pi_s(L_N)\}$ is a basis of
$\pi_s(\GG^\perp)$, and thus (\ref{dims}) holds.
Set 
$F:= \{m-\varphi(m)\mid m\in\oT^n_s\}.$
Obviously, $F_0\subseteq F\subseteq \Span_\oR(\GG_t)\cap \PR_s.$
Moreover, one can verify (cf. \cite{LLR08}) that
\begin{equation}\label{eqa}
\Span_\oR(F)=\Span_\oR(\GG)\cap \PR_s,
\end{equation}
\begin{equation}\label{eqb}
(F_0)=(F), \ \   I\subseteq (F) \ \text{ if } \ s\ge D,
\end{equation}
\begin{equation}\label{eqc}
\varphi(x_i\varphi(x_jm)) =\varphi(x_j\varphi(x_im)) \
\text{ for } m\in \BB \text{ and } i=1,\ldots,n.
\end{equation}
Note that  (\ref{ZRb}) is used to show (\ref{eqb})-(\ref{eqc}).

The ideal $J:=(F_0)$ satisfies  $I\subseteq J$ (by (\ref{eqb})) and
$J\subseteq I$ or $J\subseteq \sqrt[\oR] I$ depending on the assumption on $\GG$.  
As $\BB$ is connected to 1 and we have  the commutativity property (\ref{eqc}),
we can apply \cite[Theorem 3.1]{Mou99} and deduce that
$\BB$ is a basis of $\PR/J$.
The inclusion: $\Span_\oR(\GG)\cap \PR_s\subseteq J\cap\PR_s$ follows from 
(\ref{eqa})-(\ref{eqb}),
while the reverse inclusion follows from the fact that
$\varphi(p)=0$ for all $p\in J\cap\PR_s$ since $\BB$ is a basis of $\PR/J$.
Thus $\Span_\oR(\GG)\cap \PR_s= J\cap\PR_s$, implying
$\pi_s(\GG^\perp)= (J\cap\PR_s)^\perp$.
The inclusion $\pi_s(J^\perp)\subseteq  (J\cap\PR_s)^\perp$
is obvious, and the reverse inclusion follows from
$(\pi_s(J^\perp))^\perp\subseteq (J^\perp)^\perp\cap\PR_s= J\cap\PR_s$, 
since $J$ is zero-dimensional. Hence 
$\pi_s(\GG^\perp) = \pi_s(J^\perp)=\pi_s(\DD[J])$.
Finally note that 
$$\dim\pi_s(\GG^\perp) =|\BB|=\dim\PR/J \ge |V_\oC(J)| \ge |V_\oR(I)|.$$
Hence, if $\dim\pi_s(\GG^\perp)=|V_\oR(I)|$, then equality holds throughout, which implies that $J$ is real radical and thus $J=\sqrt [\oR]I$.
This concludes the proof of Theorem \ref{theoZR}.


\begin{remark} \label{remweak} 
We indicate here what happens if we weaken some assumptions in Theorem \ref{theoZR}.\\
 (i) The condition $s\ge D$ is used only in (\ref{eqb}) to show $I\subseteq (F)$. 
Hence if we omit the condition $s\ge D$  in Theorem \ref{theoZR},
then we get the same conclusion except that we cannot claim  $I\subseteq J$.\\
 (ii)
Consider now the case where we assume only that (\ref{ZRa}) holds (and not (\ref{ZRb})).
As we use (\ref{ZRb}) to show the existence of $\BB$ connected to 1 and to prove 
(\ref{eqb})-(\ref{eqc}), 
we cannot prove the commutativity property (\ref{eqc}), neither the equality $(F)=(F_0)$.
Nevertheless, what we can do is test whether $\BB$ is connected to 1 and whether 
(\ref{eqc}) holds. If this is the case, then
we can conclude that $\BB$ is a basis of $\PR/J$ where $J=(F_0)\subseteq \sqrt[\oR]I$ and extract the variety $V_\oC(J)$ which satisfies 
$V_\oR(I)\subseteq V_\oC(J)\cap\oR^n$ and $|V_\oC(J)|\le \dim \PR/J=|\BB|$. 
Then it suffices to sort $V_\oR(I)$ out of $V_\oC(J)$.
 The additional information that condition (\ref{ZRb}) gives us is the guarantee that the commutativity property (\ref{eqc}) holds
and the equality $J=(F)$, thus implying $J\supseteq I$ and $V_\oR(I)=V_\oC(J)\cap\oR^n$ if $s\ge D$.

\end{remark}

\subsection{A concrete choice for the polynomial system $\GG$ in Theorem \ref{theoZR}}\label{secconcrete}

For the task of computing $V_\oC(I)$, 
one can choose  as indicated in \cite{LLR08}
the set $\GG=\HH_t$ from  (\ref{setHt}) and thus consider the linear subspace
$\KK_t:=\HH_t^\perp$ of $(\PR_t)^*$.
For the task of computing $V_\oR(I)$, as inspired by \cite{LLR07},
we  augment $\HH_t$ with a set $\WW_t$ of polynomials in $\sqrt
[\oR] I$ obtained from the kernel of a 
suitable positive element in $\HH_t^\perp$. 
For this, consider the convex cone
$$\KK_{t,\succeq }:=\{L\in \HH_t^\perp\mid M_{\ttt}(L)\succeq 0\},$$
consisting of the elements of $\KK_t$ that are positive, i.e. satisfy $L(p^2)\ge 0$ whenever $\deg(p^2)\le t$.
Generic elements of $\KK_{t,\succeq}$ (defined in Lemma \ref{lemgeneric} below) play a central role; geometrically
these are the elements lying in the relative interior of the cone
$\KK_{t,\succeq}$.

\begin{lemma}\label{lemgeneric}
The following assertions are equivalent for $L^*\in \KK_{t,\succeq}$.
\begin{itemize}
\item[(i)] $\rank M_{\ttt}(L^*)=\max_{L\in \KK_{t,\succeq}} \rank M_{\ttt}(L)$.
\item[(ii)] 
$\rank M_s(L^*)=\max_{L\in \KK_{t,\succeq}} \rank M_s(L)$ for all
$1\le s\le \ttt$.
\item[(iii)] 
$\Ker M_s(L^*)\subseteq \Ker M_s(L)$ for all $L\in \KK_{t,\succeq}$ and
$1\le s\le \ttt$.
\end{itemize}
Then $L^*$ is said to be  {\em generic}.
\end{lemma}

\begin{proof}
Direct verification using (\ref{psdker})-(\ref{psdkerb}).
\end{proof}
Hence any two generic elements $L_1,L_2\in\KK_{t,\succeq}$ have the same kernel,  denoted by 
$\NNN_t$ ($=\Ker M_\ttt(L_1)=\Ker M_\ttt(L_2)$),  which satisfies
\begin{equation}\label{incNt}
\NN_t\subseteq \NN_{t'} \ \text{ if } t\le t'
\end{equation}
(easy verification), as well as 
\begin{equation}
\label{lemWt}
\NNN_t\subseteq \sqrt[\oR] I.
\end{equation}
(cf. \cite[Lemma 3.1]{LLR07}).
Define the set 
\begin{align}\label{eqn::Wt}
\WW_t:=\{\xx^\alpha g \mid \alpha\in\oN^n_\ttt, g\in \NNN_t\},
\end{align}
whose definition is motivated by the fact that, for $L\in (\PR_t)^*$, 
\begin{equation}\label{NWt}
\NNN_t\subseteq \Ker M_\ttt(L)\Longleftrightarrow 
L\in \WW_t^\perp.
\end{equation}
Therefore,  $\WW_t\subseteq \sqrt[\oR]I$.
For the task of computing $V_\oR(I)$,  our choice for the set $\GG$ in Theorem \ref{theoZR} is
\begin{equation}\label{setGt}
\GG_t:= \HH_t\cup \WW_t.
\end{equation}
Note also that
\begin{align}\label{eqn::HWt}
\KK_{t,\succeq }\subseteq \HH_t^\perp \cap \WW_t^\perp = (\HH_t\cup \WW_t)^\perp.
\end{align}
In fact, as we now show,
both sets in (\ref{eqn::HWt}) have the same dimension, i.e. $(\HH_t\cup \WW_t)^\perp$ is the smallest linear space containing the cone $\KK_{t,\succeq }$.
\begin{lemma} \label{lemdim}
$\dim \KK_{t,\succeq}= \dim (\HH_t\cup \WW_t)^\perp$.
\end{lemma}

\begin{proof}
 Pick $L^*$ lying in the relative interior of $\KK_{t,\succeq}$, i.e. $L^*$ is generic, 
 and define 
$$\PP_t:=\{L\in (\PR_t)^*\mid L^*\pm \epsilon L\in \KK_{t,\succeq} \
\text{ for some }\epsilon>0\},$$ 
the linear space consisting of all possible {\em perturbations} at $L^*$. Then, $\dim \KK_{t,\succeq}=\dim \PP_t$.
One can verify that there exists an $\epsilon>0$ such that $ L^*\pm \epsilon L\in \KK_{t,\succeq}$ if and only if 
$L\in \HH_t^\perp$ and $\Ker M_\ttt(L^*)\subseteq \Ker M_\ttt(L)$ (cf. e.g. \cite[Thm. 31.5.3]{DL97}). 
As the latter condition is equivalent to $L\in \WW_t^\perp$ by (\ref{NWt}), we find
$\PP_t= (\HH_t\cup\WW_t)^\perp$, which concludes the proof.
\end{proof}



We conclude with a characterization of $\sqrt[\oR]I$ and of its dual space $\DD[\sqrt[\oR]I]$, using the sets $\GG_t$ from (\ref{setGt}).

\begin{proposition}
\label{prop::charact} 
With $\GG_t=\HH_t\cup\WW_t$, $\sqrt[\oR]I=\bigcup_t \Span_\oR(\GG_t)$ and
$\DD[\sqrt[\oR]I]=\bigcap _t \GG_t^\perp$.
\end{proposition}

\begin{proof}
The inclusion $\bigcup_t \Span_\oR(\GG_t)\subseteq \sqrt[\oR]I$ follows from (\ref{lemWt}). Next,
for some order $(t,s)$ we have $\sqrt[\oR]I=(\Ker M_s(L^*))$. 
The proof, which relies on the existence of a finite basis for the 
ideal $\sqrt[\oR]I$ can be found in \cite{LLR07}.
This fact, combined with $\Ker M_s(L^*) \subseteq \NNN_t \subseteq \Span_\oR(\GG_t)$,
implies the reverse inclusion $ \sqrt[\oR]I\subseteq \bigcup_t \Span_\oR(\GG_t)$.
Now the equality 
$\sqrt[\oR]I =\bigcup_t\Span_\oR\GG_t$ implies in turn 
 $\DD[\sqrt[\oR]I]=\bigcap _t \GG_t^\perp$. 
\hfill\end{proof}

When $|V_\oR(I)|<\infty$,
the dual of the real radical ideal coincides in fact  with  the vector space spanned by the
 evaluations at all $v\in V_\oR(I)$. Proposition~\ref{prop::charact} shows how to obtain it directly from the quadratic forms 
$Q_L$ (or its matrix representation $M_{\lfloor t/2 \rfloor}(L)$) for a generic $L\in \KK_{t,\succeq}$ without
a priori knowledge of $V_\oR(I)$. 

\section{Links with the moment-matrix method}
\label{sec::links}
In this section we explore the links with the moment-matrix method of \cite{LLR07} for finding $V_\oR(I)$ as well as the real radical ideal $\sqrt[\oR]I$. We recall the main result of \cite{LLR07}, underlying this method.

\begin{theorem}\cite{LLR07}\label{thmrank}
Let $L^*$ be a generic element of $\KK_{t,\succeq}$.
Assume that 
\begin{align}\label{flat1}
\rank M_s(L^*)=\rank M_{s-1}(L^*)
\end{align}
for some $D\le s\le \ttt.$
Then $( \Ker M_s(L^*)) =\sqrt[\oR]I$ and any set
$\BB\subseteq \oT^n_{s-1}$ indexing a maximum linearly independent set of columns of $M_{s-1}(L^*)$ is a basis of $\PR/\sqrt[\oR]I$.
\end{theorem}

\subsection{Relating the rank condition and the prolongation-projection dimension conditions}
We now present some links between the rank condition (\ref{flat1}) and the conditions (\ref{ZRa})--(\ref{ZRb}). 
First we show that the condition (\ref{ZRa}) suffices to ensure that the rank condition (\ref{flat1}) holds at some later order.

\begin{proposition}\label{prop1}
Let $1\le s\le t$. If (\ref{ZRa}) holds with $\GG:=\HH_t\cup \WW_t$,
then $\rank M_s(L)=\rank M_{s-1}(L)$ for all $L\in \KK_{t+2s,\succeq}$.
\end{proposition}

\begin{proof}
Let $L\in \KK_{t+2s,\succeq}$. We show that $\rank M_s(L)=\rank M_{s-1}(L)$.
For this, pick $m,m'\in\oT^n_s$.
As in the proof of Theorem \ref{theoZR}, (\ref{dims}) holds and thus 
we can write 
$m= \sum_{b\in \BB}\lambda_b b +f$, where $\lambda_b\in\oR$, $f\in \Span_\oR(\GG)$, and
$\BB\subseteq \oT^n_{s-1}$.
(Note that (\ref{ZRb}) was not used to derive this.)
Then, $mm'= \sum_{b\in\BB}\lambda_b m'b + m'f.$
It suffices now to show that $L(m'f)=0$. Indeed this will imply 
$M(L)_{m',m}=L(mm')=\sum_{b\in \BB}\lambda_b L(m'b)=\sum_{b\in \BB}\lambda_b M(L)_{m',b}$,
that is, the $m$th column of $M(L)$ is a linear combination of its columns indexed by 
$b
\in\BB$, thus giving the desired result.

We now  show that $L(m'g)=0$ for all $g\in \HH_t\cup\WW_t$.
By assumption, $L\in \KK_{t+2s,\succeq}\subseteq \HH_{t+2s}^\perp \cap \WW_{t+2s}^\perp$ (recall (\ref{eqn::HWt})).
If $g\in \HH_t$, then $m'g\in \HH_{t+s}\subseteq \HH_{t+2s}$ and thus
$L(m'g)=0$.
If $g\in \WW_t$, then $g=\xx^\alpha h$,
 where $h\in \NN_t$ and $|\alpha|\le \ttt$. Hence,
$m'g = m'\xx^\alpha h$, where $\deg(m'\xx^\alpha)\le s+\ttt \le \lfloor( 2s+t)/2\rfloor$
and $h\in \NN_t \subseteq \NN_{t+2s}$ (by (\ref{incNt})), implying 
$m'g\in \WW_{t+2s}$ and thus $L(m'g)=0$.
\end{proof}

We now show that the rank condition (\ref{flat1}) is in fact equivalent to the following stronger version of the conditions (\ref{ZRa})-(\ref{ZRb}) with
$\GG=\GG_t=\HH_t\cup\WW_t$:
\begin{align}
\sublabon{equation}
\dim \pi_{2s}(\GG_t^\perp)=\dim \pi_{s-1}(\GG_t^\perp), \label{ZR++a}\\
\dim \pi_{2s}(\GG_t^\perp)=\dim \pi_{2s}((\GG_t^+)^\perp). \label{ZR++b}
\end{align}
\sublaboff{equation}
\begin{proposition}\label{prop2}
Let $L^*$ be a generic element of $\KK_{t,\succeq}$ and $1\le s\le \ttt$.
\begin{itemize}
\item[(i)] Assume (\ref{flat1}) holds. Then (\ref{ZR++a}) holds, and 
(\ref{ZR++b}) holds as well if $s\ge D$.
\item[(ii)] Assume (\ref{ZR++a})-(\ref{ZR++b}) hold. Then,
(\ref{flat1}) holds,
the ideal $J$ obtained in Theorem \ref{theoZR} is real radical and satisfies
$J=(\Ker M_s(L^*))\subseteq I(V_\oR(I))$ and,
given $\BB\subseteq \oT^n_{s-1}$, 
$\BB$ satisfies (\ref{dims-1}) if and only if $\BB$ indexes a column basis of $M_{s-1}(L^*)$. 
Furthermore, $J=\sqrt[\oR]I$ if $s\ge D$.
\end{itemize}
\end{proposition}

The proof being a bit technical is postponed to Section \ref{proofprop2}.
An immediate consequence of Proposition \ref{prop2} is that the rank condition at 
order $(t,s)$ implies the prolongation-projection dimension conditions
(\ref{ZRa})-(\ref{ZRb}) at the same order $(t,s)$.

\begin{corollary}\label{corlink}
Assume $D\le s\le \ttt$ and let $\GG=\GG_t=\HH_t\cup\WW_t$.
Then, 
$$\text{(\ref{flat1})} \Longleftrightarrow  \text{(\ref{ZR++a})-(\ref{ZR++b}) }
\Longrightarrow \text{ (\ref{ZRa})-(\ref{ZRb})}.$$
\end{corollary}

\begin{proof}
Indeed,  $\pi_s(\GG_t^\perp)=\pi_s((\GG_t^+)^\perp$ follows directly
from 
$\pi_{2s}(\GG_t^\perp)=\pi_{2s}((\GG_t^+)^\perp)$.
\end{proof}

It is shown in \cite{LLR07} that the rank condition (\ref{flat1}) holds at order $(s,t)$ large enough with $D\le s\le \ttt$. Hence
the same holds for the conditions 
(\ref{ZR++a})-(\ref{ZR++b}) (and thus for (\ref{ZRa})-(\ref{ZRb})), which
will imply 
the termination of the prolongation-projection algorithm based on Theorem \ref{theoZR}.






\subsection{Proof of Proposition \ref{prop2}} \label{proofprop2}
First we note that the rank condition (\ref{flat1}) is in fact a property of the whole cone $\KK_{t,\succeq}$ and its superset $\GG_t^\perp = \HH_{t}^\perp \cap \WW_{t}^\perp$.

\begin{lemma}\label{lemrank}
 If (\ref{flat1}) holds for some generic $L^*\in\KK_{t,\succeq}$, then 
(\ref{flat1}) holds for {\em all} $L^*\in \GG_t^\perp$.
\end{lemma}

\begin{proof}
Let $L\in \GG_t^\perp$. We have
\begin{equation}\label{kernels}
\Ker M_s(L^*) = \Ker M_\ttt(L^*)\cap \PR_s = \NN_t\cap \PR_s
\subseteq \Ker M_\ttt(L)\cap \PR_s \subseteq \Ker M_s(L),
\end{equation}
where the first equality holds by (\ref{psdker}),
the first inclusion holds by (\ref{NWt}), and the second
 one holds since $M_s(L)$ is a principal submatrix of $M_\ttt(L)$.
This implies directly that  $\rank M_s(L)=\rank M_{s-1}(L)$.
\hfill\end{proof}

We now give the proof for Proposition \ref{prop2}.
Let $L^*$ be a generic element of $\KK_{t,\succeq}$.

(i) Assume that (\ref{flat1}) holds. First we show (\ref{ZR++a}), i.e. we show that
$\dim \pi_{2s}(\GG_t^\perp)=\dim \pi_{s-1}(\GG_t^\perp)$. For this, consider
 the linear mapping
$$\begin{array}{llll}
\psi: & \pi_{2s}(\GG_t^\perp) & \rightarrow & \pi_{s-1}(\GG_t^\perp)\\
& \pi_{2s}(L) & \mapsto & \pi_{s-1}(L).
\end{array}$$
As $\psi$ is onto, it suffices to show that $\psi$ is one-to-one.
For this assume $\pi_{s-1}(L)=0$ for some $L\in \GG_t^\perp$. We show that
$\pi_{2s}(L)=0$, i.e. $L(\xx^\gamma)=0$ for all $|\gamma|\le 2s$ by induction on 
$|\gamma|\le 2s$. The case $|\gamma|\le s-1$ holds by assumption.
Let $s\le |\gamma|\le 2s$ and write $\gamma$ as $\gamma =\alpha+\beta$ where $|\alpha|=s$ and $|\beta|\le s$.
By Lemma \ref{lemrank}, $\rank M_s(L)=\rank M_{s-1}(L)$. Hence the $\alpha$th column of $M_s(L)$ can be written as a linear combination of the columns indexed by $\oT^n_{s-1}$.
This gives

$M_s(L)_{\beta,\alpha}=\sum_{|\delta|\le s-1} \lambda_\delta M_s(L)_{\beta,\delta}$ for some $\lambda_\delta\in\oR$.
As $|\beta+\delta|\le |\gamma|-1$, we have
$M_s(L)_{\beta,\delta}= L(\xx^{\beta+ \delta}) =0$ by the induction assumption, implying 
$L(\xx^\gamma)=M_s(L)_{\beta,\alpha} =0$.

We now assume moreover $s\ge D$. We show the inclusion $\pi_{2s}(\GG_t^\perp)\subseteq \pi_{2s}((\GG_t^+)^\perp)$, which implies (\ref{ZR++b}).
Let $L\in \GG_t^\perp$. As $\rank M_s(L)=\rank M_{s-1}(L)$,
we can apply Theorem \ref{theoCF} and deduce the existence of $\tilde L\in (\PR)^*$ for which
$\pi_{2s}(\tilde L)=\pi_{2s}(L)$ and 
$\Ker M(\tilde L)=(\Ker M_s(L))$.
It suffices now to show that $\tilde L\in (\GG_t^+)^\perp$. We show a
 stronger result, namely that $\tilde L\in I(V_\oR(I))^\perp$.
As $s\ge D$, we know from Theorem \ref{thmrank}
that $I(V_\oR(I))=(\Ker M_s(L^*))$.
Pick $p\in I(V_\oR(I))$ and write it as $p=\sum_l u_l g_l$, where $u_l\in\PR$ and 
$g_l\in \Ker M_s(L^*)$; we show that $\tilde L(p)=0$. 
By (\ref{kernels}),  $g_l\in \Ker M_s(L)$ and thus, as $M_s(L)=M_s(\tilde L)$,
$g_l\in \Ker M_s(\tilde L)$.
Therefore, $p$ lies in $(\Ker M_s(\tilde L))=\Ker M(\tilde L)$, which gives
$\tilde L(p)=0$.

(ii) Assume now that (\ref{ZR++a})-(\ref{ZR++b}) hold. Then, 
 (\ref{ZRa})--(\ref{ZRb}) holds for the pair $(t,2s)$ (and $\GG=\GG_t$).
 Although we do not assume $2s\ge D$, the conclusion of Theorem \ref{theoZR} partially holds, as observed in Remark~\ref{remweak}~(i).
Namely, we can find an ideal $J$ satisfying $J\subseteq I(V_\oR(I))$,
$J\cap\PR_{2s}=\Span_\oR(\GG_t)\cap \PR_{2s}$,
$\pi_{2s}(D[J])=\pi_{2s}(\GG_t^\perp)$, and $I\subseteq J$ if $2s\ge D$.
Moreover, 
there exists a set $\BB\subseteq \oT^n_{s-1}$ which is a basis of $\PR/J$ and satisfies 
the following analogue of (\ref{dims}):
\begin{equation}\label{dim2s}
\PR_{2s}= \Span_\oR(\BB) \oplus  (\Span_\oR(\GG_t)\cap \PR_{2s}). 
\end{equation}
We show that $\rank M_s(L^*)=\rank M_{s-1}(L^*)$.
As $L^*\in \GG_t^\perp$, there exists $\tilde L\in D[J]$ for which
$\pi_{2s}(L^*)=\pi_{2s}(\tilde L)$. Thus 
$M_s(L^*)=M_s(\tilde L)$, and $J\subseteq \Ker M(\tilde L)$ since
$\tilde L\in D[J]$. It suffices to show that 
$\rank M_s(\tilde L)=\rank M_{s-1}(\tilde L)$.
For this, as in the proof of Proposition \ref{prop1}, pick $m,m'\in\oT^n_s$.
Using (\ref{dim2s}), we can write
$m=\sum_{b\in \BB}\lambda_b b +f$, where $\lambda_b\in\oR$, $f\in \Span_\oR(\GG_t) \cap\PR_{2s}\subseteq J\subseteq \Ker M(\tilde L)$, so that 
 $\tilde L(m'm) =\sum_{b\in \BB}\lambda_b \tilde L(m'b)$, which gives the desired result:
$\rank M_s(\tilde L)=\rank M_{s-1}(\tilde L)$.

Let $\BB_1,\BB_2\subseteq \oT^n_{s-1}$, where $\BB_1$ 
satisfies (\ref{dims-1}) and $\BB_2$ indexes a column basis of $M_{s-1}(L^*)$.
Then
\begin{equation}\label{eq1}
|\BB_1|=\dim \PR/J \le \rank M_{s-1}(L^*) \ (=|\BB_2|) 
\end{equation}
since the columns of $M_{s-1}(L^*)$ indexed by $\BB_1$ are linearly independent
(direct verification, using (\ref{dims-1}) and the fact that 
$\Ker M_{s-1}(L^*)\subseteq\Ker M_\ttt(L^*)= \NNN_t\subseteq \Span_\oR(\GG_t)$).
Moreover,
\begin{equation}\label{eq2}
|\BB_2|=\rank M_{s-1}(L^*) \le \dim \pi_{s-1}(\GG_t^\perp)\ (=|\BB_1|).
\end{equation}
Indeed, as $\Span_\oR(\GG_t)\cap\PR_{s-1}\subseteq J\cap \PR_{s-1} \subseteq \Ker M_{s-1}(\tilde L)
=\Ker M_{s-1}(L^*)$, we obtain $\Span_\oR(\GG_t)\cap\Span_\oR(\BB_2)=\{0\}$, which implies
$|\BB_2|\le \dim (\Span_\oR(\GG_t)\cap\PR_{s-1})^\perp=\dim\pi_{s-1}(\GG_t^\perp)$.
Hence, equality holds in (\ref{eq1}) and (\ref{eq2}). 
Therefore, $\BB_1$ indexes a column basis of $M_{s-1}(L^*)$, 
$\BB_2$ satisfies (\ref{dims-1}), and 
$$\rank M_{s-1}(L^*)=\dim\pi_{s-1}(\GG_t^\perp)=\dim\PR/J.$$
As $J\subseteq \Ker M(\tilde L)$, we deduce
$$\dim \PR/\Ker M(\tilde L)\le \dim\PR/J.$$
On the other hand,
$$ \dim\PR/J =\rank M_{s-1}(L^*)=\rank M_{s-1}(\tilde L)\le \rank M(\tilde L)=\dim \PR/\Ker M(\tilde L).$$
Hence equality holds throughout.
In particular, $J=\Ker M(\tilde L)$ and 
$\rank M(\tilde L)=\rank M_{s-1}(\tilde L)$. As $M_{s-1}(\tilde L)=M_{s-1}(L^*)\succeq 0$, we deduce that $M(\tilde L)\succeq 0$ and 
$J= \Ker M(\tilde L)=(\Ker M_s(\tilde L))=(\Ker M_s(L^*))$ is a real radical ideal (using
Theorem \ref{theoA}).
Finally, if $s\ge D$, then $J=(\Ker M_s(L^*))=\sqrt[\oR]I$ by Theorem \ref{thmrank}. 
This concludes the proof of Proposition \ref{prop2}.

\subsection{Two illustrative examples}
We discuss two simple examples  to illustrate the various notions just introduced and the role of moment matrices;
the second one has infinitely many complex roots.

\begin{example} \label{exnonGor}
Let $I=( x_1^2,x_2^2,x_1x_2)\subseteq \oR[x_1,x_2]$, considered in \cite{LLR08} as an example with a non-Gorenstein algebra 
$\PR/I$.
Any $L\in \KK_t$ ($t\ge 2$) satisfies $L(\xx^\alpha)=0$ if $|\alpha|\ge 2$ and thus
$$
M_\ttt(L) = \left(\begin{matrix} a & b & c & 0& \ldots \cr
b & 0 & 0 & 0 &\ldots \cr
c & 0 & 0 & 0 &\ldots \cr
0 & 0 & 0 & 0 &\ldots \cr
\vdots & \vdots & \vdots & \vdots & \ddots \end{matrix}\right)\ \ \text{ for some scalars } a,b,c, $$
where entries are indexed by $1,x_1,x_2,\ldots$
Hence, $\dim \pi_2(\KK_2)=\dim \pi_1(\KK_2)=\dim \pi_2(\KK_3)\ =3$ 
and the rank stabilizes at order $(t,s)=(4,2)$, i.e.
$ \rank M_2(L^*)=\rank M_1(L^*)  =2$ for generic $L^*\in \KK_4$.
When $L\in\KK_{t,\succeq}$,  the condition $M_\ttt(L)\succeq 0$ implies $b=c=0$.
Hence, for generic $L^*\in\KK_{2,\succeq}$, $\NNN_2:=\Ker M_1(L^*)$ is spanned by the polynomials $x_1$ and $x_2$, and the rank condition (\ref{flat1}) holds at order
$(t,s)=(2,1)$, i.e. 
$\rank M_1(L^*)=\rank M_0(L^*) =1$.
As $\Span_\oR(\GG_2)$ is spanned by the polynomials $x_1,x_2,x_1^2,x_1x_2,x_2^2$,
the conditions 
\eqref{ZR++a}--\eqref{ZR++b} hold at the same order $(t,s)=(2,1)$, i.e.
$\dim \pi_2(\GG_2^\perp) = \dim \pi_0(\GG_2^\perp)=\dim \pi_2((\GG_2^+)^\perp) = 1,$
as predicted by Proposition \ref{prop2}.
\end{example}

\begin{example}\label{ex2}
Consider the ideal $I=( x_1^2+x_2^2)\subseteq \oR[x_1,x_2]$ with $V_\oR(I)=\{0\}$ and $|V_\oC(I)|=\infty$.
As $\dim\pi_s(\KK_t) =\dim \pi_{s-1}(\KK_t)+2$ for any $t\ge s\ge 2$, the conditions (\ref{ZRa})--(\ref{ZRb}) never hold in the case $\GG=\HH_t$.
On the other hand, any $L\in \KK_{2,\succeq}$ satisfies $L(x_1^2)=L(x_2^2)=0$,
which follows from $L(x_1^2+x_2^2)=0$ combined with $M_1(L)\succeq 0$, giving $L(x_1^2),L(x_2^2)\ge 0$. 
Moreover, $L(x_1)     =L(x_2)=L(x_1x_2)=0$. 
Thus   $\NNN_2$ is spanned by  the polynomials
 $x_1$ and $x_2$, and the
conditions (\ref{flat1}) and (\ref{ZR++a})-(\ref{ZR++b}) 
hold at order $(t,s)=(2,1)$.
\end{example}

Examples \ref{ex::cox98} and \ref{gauss} in Section \ref{sec::numexamples} are cases
where the prolongation-projection method terminates earlier than the moment-matrix method.

%
%
%
%

\section{A prolongation-projection algorithm}
\label{sec::algorithm}
Let us now give a brief description of our algorithm for computing $V_\oK(I)$ 
($\oK=\oR,\oC$)
 based on the results of the previous section.
A simple adjustment in the proposed prolongation-projection algorithm allows the computation
 of all complex vs. real roots. The general structure is shown in Algorithm~\ref{alg::numsyb_Vk}.
If $I$ is an ideal given by a set of generators and
$|V_\oK(I)|<\infty$, 
this algorithm computes the multiplication matrices in $\PR/J$, which thus allows the immediate computation of $V_\oK(J)$
(by Theorem~\ref{thm::muloperator}), 
where $J$ is a zero-dimensional ideal satisfying $J=I$ if $\oK=\oC$ and $I \subseteq J \subseteq \sqrt[\oR]{I}$ if $\oK=\oR$, so that
$V_\oK(J)=V_\oK(I)$.
We then comment on the key steps involved in the algorithm.

\begin{algorithm}
\caption{\emph{Unified prolongation-projection algorithm for computing $V_\oK(I)$:}}
\begin{algorithmic}[1]
\Require A set $\{h_1,\ldots,h_m\}$ of generators of $I$ and $t\ge D$.
\Ensure The multiplication matrices in $\PR/J$, where $J=I$ if $\oK=\oC$ and $I\subseteq J\subseteq \sqrt[\oR]I$ if $\oK=\oR$, thus enabling the computation of $V_\oK(I)$.
\State Compute the matrix representation $G_t$ of $\GG_t$ and $G_t^+$ of $\GG_t^+$.
\State Compute $\Ker G_t$ and $\Ker G_t^+$.
\State Compute $\dim \pi_s(\Ker G_t)$ ($=\dim \pi_s((\GG_t)^\perp)$) and 
$\dim \pi_s(\Ker G_t^+)$ ($=\dim \pi_s((\GG_t^+)^\perp)$)
 for $s \leq t$.
\State Check if \eqref{ZRa}--\eqref{ZRb} holds for some $D\le s\le \ttt$.
\If{yes}
\State \Return a basis $\BB\subseteq \PR_{s-1}$ connected to 1 and satisfying  (\ref{dims-1}),
and the multiplication matrices $\XX_i$ in $\PR/J$ represented in the basis $\BB$.
\Else
\State Iterate (go to 1) replacing $t$ by $t+1$.
\EndIf
\end{algorithmic}
\label{alg::numsyb_Vk}
\begin{remark} Here, $\GG_t=\HH_{t}$ (see (\ref{setHt})) for the task of computing $V_\oC(I)$,
and $\GG_t=\HH_{t}\cup \WW_t$ (see (\ref{eqn::Wt})) 
for the task of computing $V_\oR(I)$. See below for details about the matrix representations $G_t$ and $G_t^+$.
\end{remark}
\end{algorithm}


\subsubsection*{Characterizing $\GG_t$ and $\GG_t^\perp$ via the matrix $G_t$}\label{step1}

In the real case, the set $\GG_t$ is defined as $\GG_t=\HH_t\cup\WW_t$ where $\WW_t$ is the linear space defined in (\ref{eqn::Wt}). As we are interested in the orthogonal space $\GG_t^\perp$, it suffices to compute a basis $\CC_t$ of the linear space
$\NNN_t$ 
and to define the set
\begin{align}\label{eqn::St}
\SSS_t:=\{\xx^\alpha g \mid |\alpha|\le \ttt, g\in \CC_t\}.
\end{align}
Then, $\NNN_t=\Span_\oR(\CC_t)$, $\WW_t=\Span_\oR(\SSS_t)$, and 
$\GG_t^\perp= (\HH_t\cup \SSS_t)^\perp$. 
Let $S_t$ (resp., $H_t$) be the matrix with columns indexed by $\oT_t^n$ and 
whose rows are the coefficient vectors of the polynomials in  $\SSS_t$
(resp., in $\HH_t$).
In the case $\oK=\oC$, the set $\GG_t=\HH_t$ is represented by the matrix $G_t:=H_t$ and, in the case $\oK=\oR$, the set $\GG_t=\HH_t\cup\WW_t$ is represented by the matrix $$G_t := \left[\begin{array}{c}H_t \\
S_t \end{array} \right].$$
Then the vectors in $\Ker G_t$ are precisely the coordinate vectors in the canonical basis of $(\PR_t)^*$ of the linear forms in $\GG_t^\perp$, i.e.
\begin{align}\label{relker}
L\in \GG_t^\perp \Longleftrightarrow (L(\xx^\alpha))_{|\alpha|\le t}\in \Ker G_t.
\end{align}
Analogously, $G_t^+$ is the matrix representation of 
$(\HH_t\cup\SSS_t)^+$, so that $(\GG_t^+)^\perp$ corresponds to
$\Ker G_t^+$.

To compute the space $\NNN_t$ we need a generic element $L^*\in \KK_{t,\succeq}$. 
How to find such a generic element has been discussed in detail in~\cite[Section 4.4.1]{LLR07}. Let us only mention here that this task can be performed numerically using a 
standard semidefinite programming solver implementing a self-dual embedding strategy, see e.g.~\cite[Chapter 4]{K02}. 
For our computations  we use the  SDP solver SeDuMi~\cite{St99}.

\subsubsection*{Computing $\pi_s(\GG_t^\perp)$ and its dimension}

As shown in (\ref{relker}), the dual space $\GG_t^\perp$ can be characterized in the canonical dual basis as the kernel of the matrix $G_t$, see e.g. \cite{ReZh04} for details using an algorithm based on singular value decomposition. Faster implementations can be obtained e.g. using Gauss elimination. 
%
%
Once we have a basis of $\Ker G_t$, denoted say by 
 $\{z_1,\ldots,z_M\}$, then, for any $s\le t$, we construct the matrix $Z_s$ whose rows are
 the vectors $\pi_s(z_1),\ldots, \pi_s(z_M)$, the projections onto $\oR^n_s$ of $z_1,\ldots,z_M$. 
Then $\dim \pi_s(\GG_t^\perp)=\dim \pi_s(\Ker G_t)$ is equal to the rank of the matrix $Z_s$. 
%

\subsubsection*{Extracting solutions}
In order to extract the variety $V_\oK(I)$, we apply Theorem~\ref{thm::muloperator} which thus requires a basis $\BB$ of the quotient space and the corresponding multiplication matrices.
In the setting of Theorem \ref{theoZR}, $\rank Z_s=\rank Z_{s-1}=:N$ and $\BB$ is chosen such that $\BB\subseteq \oT^n_{s-1}$ indexes $N$ linearly independent columns of $Z_{s-1}$.
A first possibility to construct $\BB$ is to use a greedy algorithm as explained in the proof of Lemma \ref{lemstable}. 
Another possibility is to use Gauss-Jordan elimination with partial pivoting on $Z_s$ (see~\cite{HeLa05}) such that each column corresponding to a monomial of degree $s$ is expressed as a linear combination of $N$ monomials of degree at most $s-1$.
 The pivot variables form a set $\BB\subseteq \oT^n_{s-1}$ indexing a maximum set of linearly independent columns of $Z_s$ and their corresponding monomials serve as a (monomial) basis $\BB$ of the quotient space (provided $\BB$ is connected to 1). 
The reduced row echelon form of $Z_s$, interpreted as coefficient vector for some polynomials, gives the desired rewriting family, which thus enables the construction of multiplication matrices and provides a border (or Gr\"obner) basis (cf. \cite{LLR07} for details).

A second alternative proposed in~\cite{ReZh04} is to use singular value decomposition once more to obtain a basis of $\Ker Z_s$ and therefore a polynomial basis $\BB$ for the quotient ring (see~\cite{ReZh04} for details). All examples presented in the next section are computed using singular value decomposition.

\section{Numerical Examples}
\label{sec::numexamples}
We now illustrate the prolongation-projection algorithm on some simple examples. The algorithm has been implemented in Matlab using the Yalmip toolbox~\cite{YALMIP}. For the real-root prolongation-projection algorithm,
we show the dimensions of $\pi_s(\GG_t^\perp)$ and $\pi_s((\GG_t^+)^\perp)$, the projections 
of the orthogonal complement of the set $\GG_t = \HH_t \cup \WW_t$ and of its one degree prolongation. For comparison, we also sometimes show the dimension table for the complex-root
version of this algorithm, and we show the values $\rank{M_s(L^*)}$ ($s\le \ttt$)
for a generic element $L^*\in \KK_{t,\succeq}$ used in the real moment-matrix method. To illustrate the potential savings, and at the same time facilitate the comparison between  the various methods, we sometimes give more data than needed for the real root computation (then displayed in color gray). We also provide the extracted roots $v \in V_\oK(I)$ and,
as a measure of accuracy, the maximum evaluation  $\epsilon(v) = \max_{j} |h_j(v)|$ taken over all input polynomials $h_j$ at the
 extracted root $v$, as well as the commutativity error $c(\XX):= \max_{i,j=1}^n \text{abs}(\XX_{i}\XX_{j}-\XX_{j}\XX_{i})$ of the computed multiplication matrices $\XX_{i}$.


\begin{example}
\label{ex::cox98}
Consider the ideal $I =(h_1,h_2,h_3)\subseteq \oR[x_1,x_2,x_3]$, where
\begin{align*}
h_1&=x_1^2-2x_1x_3+5\,,\\
h_2&=x_1x_2^2+x_2x_3+1\,,\\
h_3&=3x_2^2-8x_1x_3\,,
\end{align*}
with $D=3$, $|V_\oC(I)|=8$ and $|V_\oR(I)|=2$, 
taken from~\cite[Ex. 4, p.57]{CLO98}. 
We illustrate and compare the various algorithms on this example.

\smallskip
\begin{table}[htpb]
	\centering
	
	\begin{tabular}{c|cccccccccc}
\multicolumn{1}{r}{ \hfill $s\;=$} & 0 & 1 & 2 & 3 & 4 & 5&6&7&8&9 \\
\hline 
\hline
$\dim \pi_s(\KK_3)$   & 1 & 4 & 8 & 11 & --- & --- &---&---&---&---\\
\hline 
$\dim \pi_s(\KK_4)$   & 1 & 4 & 8 & 10 & 12 & ---  & ---&---&---&---\\
\hline
$\dim \pi_s(\KK_5)$   & 1 & 4 & 8 & 9 & 10 & 12 &---&---&---&---\\
\hline 
$\dim \pi_s(\KK_6)$   & 1 & 4 & \textbf{8} & \textbf{8} &  \color{medgray}9 &  \color{medgray}10  &  \color{medgray} 12 &---&---&---\\
\hline
$\dim \pi_s(\KK_7)$   & 1 & 4 & 8 & \textbf{8} & \color{medgray}8 & \color{medgray}9  & \color{medgray}10 & \color{medgray}12&---&---\\
\hline
\color{medgray} $\dim \pi_s(\KK_8)$   & \color{medgray} 1 & \color{medgray} 4 & \color{medgray} 8 & \color{medgray} 8 & \color{medgray} 8 & \color{medgray} 8  & \color{medgray} 9 & \color{medgray} 10 &\color{medgray} 12&---\\
\hline
\color{medgray} $\dim \pi_s(\KK_9)$   & \color{medgray} 1 & \color{medgray} 4 & \color{medgray} 8 & \color{medgray} 8 & \color{medgray} 8 & \color{medgray} 8  &\color{medgray} 8  & \color{medgray} 9 & \color{medgray} 10 &\color{medgray} 12\\
\hline
\end{tabular}
\caption{Dimension table for $\pi_s(\KK_t)$ in Example~\ref{ex::cox98}.\label{tab::complexalg}}
\end{table}

Table~\ref{tab::complexalg} shows the dimensions of the sets $\pi_s(\KK_t)$ for various
prolongation-projection orders $(t,s)$.
Note that the conditions (\ref{ZRa})--(\ref{ZRb})
hold at order $(t,s)=(6,3)$, i.e. 
\begin{align*}
\dim \pi_3(\KK_6)=\dim \pi_{2}(\KK_6)=\dim \pi_{3}(\KK_7).
\end{align*}
With the complex-root prolongation-projection algorithm we can 
 compute the 
 following eight  complex roots:
\begin{align*}
v_1 &=\left[\begin{array}{ccc} -1.10 &    -2.88 &    -2.82 \end{array} \right] \,,\\
v_2 &=\left[\begin{array}{ccc} 0.0767+2.243i      &    0.461+0.497i   &     0.0764+0.00834i  \end{array} \right]\,,\\
v_3 &=\left[\begin{array}{ccc} 0.0767-2.243i       &   0.461-0.497i     &   0.0764-0.00834i  \end{array} \right]\,,\\
v_4 & =\left[\begin{array}{ccc} -0.0815-0.931i   &       2.35+0.0431i     &    -0.274+2.209i     \end{array}\right]\,,\\
v_5 & =\left[\begin{array}{ccc} -0.0815+0.931i    &      2.35-0.0431i    &     -0.274-2.20i     \end{array} \right]\,, \\
v_6 & =\left[\begin{array}{ccc} 0.0725+2.24i     &     -0.466-0.464i    &    0.0724+0.00210i   \end{array} \right] \,,\\
v_7 &=\left[\begin{array}{ccc}  0.0725-2.24i     &     -0.466+0.464i    &    0.0724-0.00210i   \end{array} \right] \,,\\
v_8 &=\left[\begin{array}{ccc} 0.966   &  -2.81  &    3.07 \end{array} \right]\,,
\end{align*}
with a maximum error of $\max_i \epsilon(v_i) < 8 \text{e-13}$ and commutativity error $c(\XX) <\text{6e-13}$.
%
%
%

\smallskip
\begin{table}[htpb]
\centering
\begin{tabular}{c|cccccccc}
\multicolumn{1}{r}{ \hfill $s\;=$} & 0 & 1 & 2 & 3 & 4 & 5&6 &7\\
\hline 
\hline
$\dim \pi_s(\GG_3^\perp)$   & 1 & 4 & 8 & 11 & --- & --- &---&---\\
$\dim \pi_s((\GG_3^+)^\perp)$ & 1 & 4 & 8 & 10 & 12 & ---  & ---&---\\
\hline 
$\dim \pi_s(\GG_4^\perp)$   & 1 & 4 & 8 & 10 & 12 & --- &---&---\\
$\dim \pi_s((\GG_4^+)^\perp)$ & 1 & 4 & 8 & 9 & 10 & 12 & ---&---\\
\hline 
$\dim \pi_s(\GG_5^\perp)$   & 1 & \textbf{2} & \textbf{2} & \color{medgray}2 &\color{medgray} 3 &\color{medgray} 5 &---&---\\
$\dim \pi_s((\GG_5^+)^\perp)$ & 1 & 2 & \textbf{2} & \color{medgray}2 & \color{medgray}3 & \color{medgray}4  & \color{medgray}6&---\\
\hline
$\dim \pi_s(\GG_6^\perp)$    &\color{medgray}1&  \color{medgray}2 & \color{medgray}2  &\color{medgray}2  &\color{medgray}2  &\color{medgray}2  &\color{medgray}3&---\\
$\dim \pi_s((\GG_6^+)^\perp)$ & \color{medgray}1&  \color{medgray}2&  \color{medgray}2 & \color{medgray}2 & \color{medgray}2 & \color{medgray}2 &\color{medgray} 2 & \color{medgray}3\\
\hline
\end{tabular}
\caption{Dimension table for $\pi_s(\GG_t^\perp)$ and $\pi_s((\GG_t^+)^\perp)$ in Example~\ref{ex::cox98}.\label{tab::realalg}}
\end{table}

Table~\ref{tab::realalg} shows the dimensions of the sets $\pi_s(\GG_t^\perp)$ and $\pi_{s}((\GG^+_t)^\perp)$ for various prolongation-projection orders
$(t,s)$.
Note that the conditions \eqref{ZRa}--\eqref{ZRb} hold at order $(t,s)=(5,2)$, i.e. 
\begin{align*}
\dim \pi_{2}(\GG_5^\perp)=\dim \pi_{1}(\GG_5^\perp)=\dim \pi_{2}((\GG^+_5)^\perp).
\end{align*}
With the real-root prolongation-projection algorithm we can 
 extract the two real solutions:
\begin{align*}
v_1 &=\left[\begin{array}{ccc} -1.101  &    -2.878 &    -2.821 \end{array} \right]\,, \\
v_2 &=\left[\begin{array}{ccc} 0.966  &   -2.813  &    3.072 \end{array} \right]\,,
\end{align*}
with $\max_i \epsilon(v_i) < 2 \text{e-8}$ and commutativity error $c(\XX) <\text{3.3e-9}$.
Note that, since $2=s<D=3$, we cannot directly apply Theorem~\ref{theoZR} 
to claim $V_\oR(I)=V_\oC(J)\cap\oR^n$.
Instead, as indicated in Remark~\ref{remweak}~(i), we can only claim  $V_\oC(J)\cap\oR^n  \supseteq V_\oR (I)$.
However, equality can be verified by 
evaluating  the input polynomials $h_j$ at the points $v \in V_\oC(J)\cap\oR^n$.
Anyway, one can also observe that the conditions \eqref{ZRa}--\eqref{ZRb} hold at order
$(t,s)=(5,3)$, in which case one can directly conclude  $V_\oR(I)=V_\oC(J)\cap\oR^n$.
Finally, we can even conclude $J=\sqrt[\oR]I$ since $\dim\pi_s(\GG_t^\perp)=|V_\oR(I)|$ (using the last claim in Theorem \ref{theoZR}).

The ranks of the moment matrices involved in the computation are shown in Table~\ref{tab::excox_rank}. Observe that the rank condition (\ref{flat1}) holds at order $(t,s)=(6,2)$, i.e.
$$\rank M_2(L^*)=\rank M_1(L^*) \quad \text{ for generic } L^*\in\KK_{6,\succeq}.$$
 (To be precise, as $2=s<D=3$, we use \cite[Prop. 4.1]{LLR07} and check whether 
the extracted roots belong to $V_\oR(I)$ afterwards.)
\begin{table}[ht!]
\begin{tabular}{c||ccccc}
 & $s=0$ & $s=1$ & $s=2$& $s=3$\\
\hline 
\hline
$t = 3$   & 1 & 4 & ---& ---\\
\hline
$t=4$     & 1 & 4 & 8  & ---\\
\hline
$t=5$     & 1 & 2 & 8  & ---\\
\hline
$t=6$     & 1 & \textbf{2} & \textbf{2}  & \color{medgray} 10  \\
\hline 
\end{tabular}
\caption{Showing $\rank M_{s}(L^*)$ for generic $L^*\in \KK_{t,\succeq}$ in Example~\ref{ex::cox98}.\label{tab::excox_rank}}
\end{table}

In this small example, we see that we can improve efficiency over the general complex-root algorithm if we are only interested in computing the real roots. 
Indeed the prolongation-projection algorithm terminates at order $(t,s)=(5,2)$ in the real case while it terminates at order $(6,3)$ in the complex case, however at the 
price of solving an SDP in the real case.
Moreover, compared to the real-root moment-matrix algorithm of \cite{LLR07}, we save the computation of the last moment matrix $M_3(L^*)$ for $L^* \in\KK_{6,\succeq}$. 

Modifying the above example by replacing each polynomial $h_i$ by $h_i\cdot(1+ \sum_i x_1^2)$ yields an example with a positive
dimensional complex variety, while the real variety is unchanged. The proposed algorithm still converges, this time at order $(t,s) = (7,2)$ and allows the extraction of the two real roots.
\end{example}

\begin{example}\label{cox3}
Consider the ideal $I=(h_1,h_2,h_3)\subseteq \oR[x_1,x_2]$, where
\begin{align*} h_1&= x_2^4 x_1+3 x_1^3-x_2^4-3 x_1^2,\\
h_2&=x_1^2 x_2-2 x_1^2,\\
h_3&=2 x_2^4 x_1-x_1^3-2 x_2^4+x_1^2,
\end{align*}
and $D=5$, taken from \cite[p.40]{CLO98}.
The corresponding variety consists of two (real) points, one of which has
 multiplicity 8.

\begin{table}[ht!]
\begin{tabular}{c|ccccccccc}
\multicolumn{1}{r}{ \hfill $s\;=$} & 0 & 1 & 2 & 3 & 4 & 5&6&7 \\
\hline 
\hline
$\dim \pi_s(\GG_5^\perp)$   & 1 & 3 & 5 & 6 & 8 & 10 &---&---\\
$\dim \pi_s((\GG_5^+)^\perp)$ & 1 & 3 & 5 & 6 & 6 & 8  & 10&---\\
\hline 
$\dim \pi_s(\GG_6^\perp)$   & 1 & \textbf{2} & \textbf{2} & \textbf{2} & \textbf{2} & \textbf{2}  & \color{medgray}4 &---\\
$\dim \pi_s((\GG_6^+)^\perp)$ & 1 & 2 & \textbf{2} & \textbf{2} & \textbf{2} & \textbf{2}  & \color{medgray}2 & \color{medgray}4 \\
\hline
\end{tabular}
\caption{Dimension table for $\pi_s(\GG_t^\perp)$ and $\pi_s((\GG_t^+)^\perp)$ in Example~\ref{cox3}.\label{tab::ex1_dimG}}
\end{table}
\noindent
Table~\ref{tab::ex1_dimG} shows the dimensions of the projections of the sets $\GG_t^\perp$ and $(\GG_t^+)^\perp$.
The conditions \eqref{ZRa}--\eqref{ZRb} hold at order $(t,s)=(6,s)$ with $2\le s\le 5$, i.e.
\begin{align*}
\dim \pi_s(\GG_6^\perp)=\dim \pi_{s-1} (\GG_6^\perp) =\dim \pi_s((\GG^+_6)^\perp) \
\ \text{ for } 2\le s\le 5,
\end{align*} 
the conditions~\eqref{ZR++a}--\eqref{ZR++b} hold at order $(t,s)=(6,2)$, i.e.
\begin{align*}
\dim \pi_{1}(\GG_6^\perp)=\dim \pi_{4}(\GG_6^\perp)=\dim \pi_{4}((\GG^+_6)^\perp),
\end{align*}
and the extracted roots are
\begin{align*}
v_1&=\left[\begin{array}{cc}-6.17\text{e-6}&1.10\text{e-5}\end{array} \right]\\
v_2&=\left[\begin{array}{cc}0.9988 & 1.9998\end{array} \right]
\end{align*}
with an accuracy of $\epsilon(v_1) < 2\text{e-10}$ and $\epsilon(v_2) < 4\text{e-3}$ and maximum commutativity error $c(\XX) <3\text{e-5}$.

\noindent
The ranks of the moment matrices involved in the computations are shown in Table~\ref{tab::ex1_rank}. 
As predicted by Proposition~\ref{prop2}, condition (\ref{flat1}) 
holds at order $(t,s)=(6,2)$, i.e. 
$$\rank M_2(L^*)=\rank M_1(L^*)\quad \text{  for generic } L^*\in \KK_{6,\succeq}.$$
Moreover, the returned ideal $J$ satisfies $J=(\Ker M_1(L^*))=\sqrt[\oR]I$.
\begin{table}[ht!]
\begin{tabular}{c||cccc}
 & $s=0$ & $s=1$ & $s=2$& $s=3$\\
\hline 
\hline
$t = 5$   & 1 & 3 & 5 &---\\
\hline
$t=6$ & 1 & \textbf{2} & \textbf{2} & 4 \\
\hline 
\end{tabular}
\caption{Showing $\rank M_s(L^*)$ for generic $L^*\in \KK_{t,\succeq}$ in Example~\ref{cox3}.\label{tab::ex1_rank}}
\end{table}
Table~\ref{tab::ex1_dimK} shows the dimensions of the projections $\pi_s(\KK_t)$ for the complex-root prolongation-projection algorithm.
The conditions~\eqref{ZRa}--\eqref{ZRb} are satisfied at order $(t,s) = (7,5)$, allowing (in principle) to extract the two roots with their corresponding multiplicities. The appearance of multiple roots requires a careful choice of the extraction procedure using multiplication operators. We employ the approach described in~\cite{CPGT97} using reordered Schur factorization. At order $(t,s) = (7,5)$, numerical problems prevent a successful extraction despite this algorithm. However, at order $(t,s) = (8,5)$, the multiplication matrices (on which the reordered Schur factorization method is applied) have a commutativity error of $c(\XX)<6.25\text{e-16}$.
Thus, we can extract the root
\begin{align*}
v &=\left[\begin{array}{cc} 1  &    2 \end{array} \right]\, 
\end{align*}
with accuracy $\epsilon(v) < 1.38\text{e-14}$ and the 8-fold root at the origin with an even higher accuracy of 
$\epsilon(v_i) < 1.75\text{e-32}$.

Note that the real version of this algorithm, working directly with the real radical of the ideal, does not require these considerations as it eliminates multiplicities.
\begin{table}[ht!]
\begin{tabular}{c|ccccccccccc}
\multicolumn{1}{r}{ \hfill $s\;=$} & 0 & 1 & 2 & 3 & 4 & 5&6&7&8 &9&10\\
\hline 
\hline
$\dim \pi_s(\KK_5)$   & 1 & 3 & 6 & 8 & 11 & 13 &---&---&---&---&---\\
\hline 
$\dim \pi_s(\KK_6)$   & 1 & 3 & 6 & 8 & 9 & 11  & 13 &---&---&---&---\\
\hline
$\dim \pi_s(\KK_7)$   & 1 & 3 & 6 & 8 & \textbf{9} & \textbf{9}  & \color{medgray} 11 & \color{medgray} 13&---&---&---\\
\hline
$\dim \pi_s(\KK_8)$   & 1 & 3 & 6 & 8 & 9          & \textbf{9}  & \color{medgray}9 & \color{medgray}11 & \color{medgray}13&---&---\\
\hline
$\dim \pi_s(\KK_9)$   &  \color{medgray} 1 & \color{medgray} 3 & \color{medgray} 6 & \color{medgray} 8 & \color{medgray} 9 & \color{medgray} 9 & \color{medgray} 9 & \color{medgray} 9 & \color{medgray} 11 & \color{medgray} 13&---\\
\hline
$\dim \pi_s(\KK_{10})$  &  \color{medgray} 1 & \color{medgray} 3 & \color{medgray} 6 & \color{medgray} 8 & \color{medgray} 9 & \color{medgray} 9 & \color{medgray} 9 & \color{medgray} 9 & \color{medgray} 9 & \color{medgray} 11 & \color{medgray} 13\\
\hline
\end{tabular}
\caption{Dimension table for $\pi_s(\KK_t)$ in Example~\ref{cox3}.\label{tab::ex1_dimK}}
\end{table}

\end{example}

\begin{example} \label{gauss}
This example is taken from \cite{VeGa95} and 
represents a Gaussian quadrature formula with two 
weights and two knots, namely,  $I=( h_1,\ldots,h_4)$, where
 \begin{align*} h_1& = x_1 + x_2 - 2 \,,\\
 h_2&=x_1x_3 + x_2 x_4 \,,\\
 h_3&=x_1x_3^2 + x_2 x_4^2 - \frac{2}{3}  \,,\\
 h_4&=x_1 x_3^3 + x_2 x_4^3, 
\end{align*}
with $D=4$ and $|V_\oR(I)| = |V_\oC(I)| = 2$.
Table~\ref{tab::ex2_dimG} shows the dimensions for the projections of the sets $\GG_t^\perp$ and $(\GG_t^+)^\perp$
and  Table~\ref{tab::ex2_rank} shows the ranks of the moment matrices $M_s(L^*)$ for generic $L^*\in \KK_{t,\succeq}$.
The conditions~\eqref{ZRa}--\eqref{ZRb} hold at order $(t,s)=(5,2)$ and the extracted roots are
\begin{align*}
v_1&=\left[\begin{array}{cccc}1& 1& -0.5774& 0.5774 \end{array} \right] \\
v_2&=\left[\begin{array}{cccc}1& 1& 0.5774& -0.5774 \end{array} \right].
\end{align*}
with an accuracy of $\epsilon(v_1) < 2\text{e-11}$ and $\epsilon(v_2) < 2\text{e-11}$ and maximum commutativity error $c(\XX) <\text{4e-14} $. 
Here again the algorithm returns the ideal $J=\sqrt[\oR]I$, since $\dim\pi_2(\GG_5^\perp) = |V_\oR(I)|=2$.
On the other hand, the moment-matrix algorithm of \cite{LLR07} terminates at order $(t,s)=(6,2)$, thus later than the prolongation-projection algorithm.

\begin{table}[ht!]
\begin{tabular}{c|cccccccc}
\multicolumn{1}{r}{ \hfill $s\;=$} & 0 & 1 & 2 & 3 & 4 & 5&6 &7\\
\hline 
\hline
$\dim \pi_s(\GG_4^\perp)$   & 1 & 3 & 7 & 11 & 20 & --- &---& ---\\
$\dim \pi_s((\GG_4^+)^\perp)$ & 1 & 3 & 4 &  8 & 12 & 23 & ---& ---\\
\hline 
$\dim \pi_s(\GG_5^\perp)$   & 1 & \textbf{2} & \textbf{2} &\color{medgray} 2 &\color{medgray} 5 & \color{medgray} 16  & ---& ---\\
$\dim \pi_s((\GG_5^+)^\perp)$ & 1 & 2 & \textbf{2} &\color{medgray}2 & \color{medgray}5 & \color{medgray}9  & \color{medgray}22& ---\\
\hline
$\dim \pi_s(\GG_6^\perp)$ &\color{medgray}1  & \color{medgray}2 & \color{medgray}2 & \color{medgray} 2  &\color{medgray} 2  &\color{medgray}16 & \color{medgray}18& ---\\
 $\dim \pi_s((\GG_6^+)^\perp)$ &\color{medgray}1 &\color{medgray} 2 & \color{medgray}2 &\color{medgray} 2 &\color{medgray} 2 &\color{medgray} 2 &\color{medgray} 2 &\color{medgray} 2\\
 \hline
\end{tabular}
\caption{Dimension table for $\pi_s(\GG_t^\perp)$ and $\pi_s((\GG_t^+)^\perp)$ in Example~\ref{gauss}.\label{tab::ex2_dimG}}
\end{table}
\begin{table}[ht!]
\begin{tabular}{c||ccccc}
 & $s=0$ & $s=1$ & $s=2$&$s=3$\\
\hline 
\hline
$t = 4$   & 1 & 4 & 9 &---\\
\hline
$t=5$     & 1 &2 &5&--- \\
\hline 
 $t=6$ & \color{medgray}1 &\color{medgray}\textbf{2} &\color{medgray}\textbf{2}&\color{medgray}9 \\
\hline
\end{tabular}
\caption{Showing $\rank M_s(L^*)$ for generic  $L^*\in \KK_{t,\succeq}$ in Example~\ref{gauss}.\label{tab::ex2_rank}}
\end{table}
\end{example}

\begin{example}
\label{katsura5}
The following 6-dimensional system is taken from
\begin{center}
\verb1http://www.mat.univie.ac.at/~neum/glopt/coconut/Benchmark/Library3/katsura5.mod1 
\end{center} 
and is known under the name Katsura 5:
\begin{align*} h_1&=2 x_6^2+2 x_5^2+2 x_4^2+2 x_3^2+2 x_2^2+x_1^2-x_1\,,\\
h_2&=x_6 x_5+x_5 x_4+2 x_4 x_3+2 x_3 x_2+2 x_2 x_1-x_2 \,,\\
h_3&=2 x_6 x_4+2 x_5 x_3+2 x_4 x_2+x_2^2+2 x_3 x_1-x_3\,,\\
h_4&=2 x_6 x_3+2 x_5 x_2+2 x_3 x_2+2 x_4 x_1-x_4 \,,\\
h_5&=x_3^2+2 x_6 x_1+2 x_5 x_1+2 x_4 x_1-x_5 \,,\\
h_6&=2 x_6+2 x_5+2 x_4+2 x_3+2 x_2+x_1-1,
\end{align*}
with $D=2$, $|V_\oC(I)|=32$, and $|V_\oR(I)|=12$.
The projection dimensions  are shown in Table~\ref{tab::ex_katasura5}.

\begin{table}[ht!]
\begin{tabular}{c|cccccccc}
\multicolumn{1}{r}{ \hfill $s\;=$} & 0 & 1 & 2 & 3 & 4 & 5&6 &7\\
\hline 
\hline
$\dim \pi_s(\GG_2^\perp)$     & 1 & 6 & 16 & --- & --- & --- &---& ---\\
$\dim \pi_s((\GG_2^+)^\perp)$ & 1 & 6 & 16 & 26 & --- & --- & ---& ---\\
\hline 
$\dim \pi_s(\GG_3^\perp)$     & 1 & 6 & 16 & 26 & --- & --- &---& ---\\
$\dim \pi_s((\GG_3^+)^\perp)$ & 1 & 6 & 16 &  26 & 31 & --- & ---& ---\\
\hline 
$\dim \pi_s(\GG_4^\perp)$     & 1 & 6 & 16 & 26 & 31 & --- &---& ---\\
$\dim \pi_s((\GG_4^+)^\perp)$ & 1 & 6 & 16 & 26 & 31 & 32 & ---& ---\\
\hline 
$\dim \pi_s(\GG_5^\perp)$     & 1 & 6 & 16 & 26 & 31 & 32 &---& ---\\
$\dim \pi_s((\GG_5^+)^\perp)$ & 1 & 6 & 16 & 26 & 31 & 32 & 32 & ---\\
\hline 
$\dim \pi_s(\GG_6^\perp)$     & 1 & 6 & \textbf{12} & \textbf{12} & \color{medgray} 12& \color{medgray} 12& \color{medgray} 12& ---\\
$\dim \pi_s((\GG_6^+)^\perp)$ & 1 & 6 & 12 &\textbf{12}  & \color{medgray} 12 &\color{medgray} 12  & \color{medgray} 12&\color{medgray} 12\\
\hline
\end{tabular}
\caption{Dimension table for $\pi_s(\GG_t^\perp)$ and $\pi_s((\GG_t^+)^\perp)$ in Example~\ref{katsura5}.\label{tab::ex_katasura5}}
\end{table}

The extracted solution points
\begin{align*}
v_1&=(1, 8.73\text{e-7}, 2.14\text{e-6},  2.48\text{e-7},  2.23\text{e-6}, -1.29\text{e-6})\,,\\
v_2&=(0.277,0.226,0.162,0.0858,0.0115,-0.124)\,,\\
v_3&=(0.136,0.0428,0.0417,0.0404,0.0964,0.211)\,,\\
v_4&=(0.462,0.309,0.0553,-0.102,-0.0844,0.0917)\,,\\
v_5&=(0.441,0.151,0.0225,0.219,0.0935,-0.207)\,,\\
v_6&=(0.239,0.0608,-0.0622,-0.0233,0.186,0.219)\,,\\
v_7&=(0.753,0.0532,0.191,-0.114,-0.146,0.139)\,,\\
v_8&=(0.726,-0.0503,0.122,0.164,0.109,-0.208)\,,\\
v_9&=(0.409,-0.0732,0.0657,-0.127,0.252,0.178)\,,\\
v_{10}&=(0.292,-0.101,0.181,-0.0591,0.193,0.141)\,,\\
v_{11}&=(0.590,0.0422,0.327,-0.0642,-0.0874,-0.0132)\,,\\
v_{12}&=(0.68,0.266,-0.154,0.0323,0.0897,-0.0735)\,,
\end{align*}
were extracted at order $(t,s) = (6,3)$, when conditions~\eqref{ZRa}--\eqref{ZRb} were first satisfied. The maximum evaluation error was found to be $\max_i \epsilon(v_i) < 2.4 \text{e-4}$ and the commutativity error $c(\XX) < \text{6.2e-6}$.
Again the algorithm returns the ideal $J=\sqrt[\oR ]I$ as
$\dim\pi_3(\GG_6^\perp)=|V_\oR(I)|=12$.
In this example the moment-matrix method \cite{LLR07} also extracts the 12 real solutions at order $(t,s)=(6,3)$.
\end{example}

\section{Conclusion}
This work was motivated by the great success of numerical-algebraic methods in recent years. Incorporating features specific to \emph{real root finding} into efficient symbolic-numeric methods may lead to more efficient algorithms for numerically computing all real roots of a given system of polynomials. The contribution of this paper is a first attempt in this direction as it implements real-algebraic
features into the existing symbolic-numeric algorithm described in~\cite{ReZh04}. Concretely, the resulting algorithm
uses semidefinite programming techniques in addition to standard numerical linear algebra techniques. It is not only applicable to zero-dimensional ideals, but to all problems for which the real variety is finite. An extension to zero-dimensional basic semi-algebraic subsets is also possible, along the same lines as in \cite{LLR07}.
%


The new approach relies on a dual space characterization of (an approximation of) the real radical ideal, obtained by combining ideas of \cite{LLR07} and \cite{ReZh04}, but the new prolongation-projection algorithm may terminate earlier than the moment-matrix method of \cite{LLR07}. Although preliminary computational results are encouraging,
whether the characterization at hand can lead to a new
treatment of real-algebraic problems is still to be demonstrated on a larger sample of problems.
An important computational issue is how to efficiently solve the underlying semidefinite program for large problems involving high degree polynomials with many variables. Exploiting sparsity in order to decrease the size of the semidefinite program is a promising direction and the work of Kojima et al. \cite{KoKiWa05} and Lasserre~\cite{Las06a} is a first important step in this direction. Strategies similar to those used in Gr\"obner/border basis computations can be employed to further increase efficiency of the proposed method, particularly in view of the linear algebra steps involved, e.g. the dimension tests.

\section*{Acknowledgements}
We thank two referees for their careful reading and useful suggestions which helped improve the presentation of the paper.

\bibliographystyle{abbrv}

\end{document}